\documentclass[12pt,thmsa, reqno]{amsart}

\usepackage{amsmath} 
\usepackage{amssymb, euscript, amsthm, amsfonts} 

\usepackage{graphicx} 
\usepackage{multicol}
\usepackage{amscd,mathrsfs}
\usepackage{caption}

\def\th{\theta}

\def\Cal{\mathcal}

\def\D{{\EuScript{D}}}

\def\I{{\Cal I}}

\def\T{{\Cal T}}

\def\G{\mathcal{G}}

\def\Q{\mathcal{Q}}

\def\f0{f_0}
\def\Fc0{\varphi_0}
\def\rn{\bbr^n}

\def\I_k {I_{-}^{k/2}}
\def\I+k {I_{+}^{k/2}}

\def\bbr{{\Bbb R}}

\def\bbn{{\Bbb N}}
\def\bbh{{\Bbb H}}

\def\bbc{{\Bbb C}}
\def\bbz{{\Bbb Z}}

\def\bbe{{\Bbb E}}
\def \hn{{\bbh^n}}

\def\supp{{\hbox{\rm supp}}}

\def\const{{\hbox{\rm const}}}
\def\cos{{\hbox{\rm cos}}}

\def\ch{{\hbox{\rm cosh}}}
\def\sh{{\hbox{\rm sinh}}}

\def \hns {\overset {*}{\bbh}{}^n}

\def\part{\partial}
\def\intl{\int\limits}

\def\Gam{\Gamma}

\def\a{\alpha}
\def\om{\omega}

\def\del{\delta}
\def\vp{\varphi}

\def\gam{\gamma}

\def\sig{\sigma}
\def\lam{\lambda}

\def\e{\varepsilon}
\def\t{\tau}

\def\sn{S^{n-1}}

\font\frak=eufm10

\def\fr#1{\hbox{\frak #1}}

\def\frR{\fr{R}}        
\def\frS{\fr{S}}

\def\const{{\hbox{\rm const}}}
\def\cos{{\hbox{\rm cos}}}

\def\part{\partial}
\def\intl{\int\limits}

\def\Gam{\Gamma}

\def\a{\alpha}

\def\hn{\bbh^n}

\newtheorem{theorem}{Theorem}[section]
\newtheorem{lemma}[theorem]{Lemma}
\newtheorem{definition}[theorem]{Definition}

\newtheorem{corollary}[theorem]{Corollary}
\newtheorem{proposition}[theorem]{Proposition}

\theoremstyle{remark}
\newtheorem{remark}[theorem]{Remark}

\numberwithin{equation}{section}

\newcommand{\be}{\begin{equation}}
\newcommand{\ee}{\end{equation}}

\newcommand{\bea}{\begin{eqnarray}}
\newcommand{\eea}{\end{eqnarray}}
\newcommand{\Bea}{\begin{eqnarray*}}
\newcommand{\Eea}{\end{eqnarray*}}

\def\sideremark#1{\ifvmode\leavevmode\fi\vadjust{\vbox to0pt{\vss
 \hbox to 0pt{\hskip\hsize\hskip1em
\vbox{\hsize2cm\tiny\raggedright\pretolerance10000
 \noindent #1\hfill}\hss}\vbox to8pt{\vfil}\vss}}}%

\begin{document}

\title[ Null Spaces of  Radon Transforms]
{Null Spaces of Radon Transforms}

\author{Ricardo  Estrada}
\address{Department of Mathematics, Louisiana State University, Baton Rouge,
LA, 70803 USA} \email{restrada@math.lsu.edu}

\author{Boris Rubin }
\address{Department of Mathematics, Louisiana State University, Baton Rouge,
Louisiana 70803, USA}
\email{borisr@math.lsu.edu}






\begin{abstract}

We obtain new  descriptions of the null spaces of several projectively equivalent transforms  in integral geometry.  The paper deals with the  hyperplane Radon transform, the totally geodesic transforms on the sphere and the hyperbolic space, the spherical slice transform, and the Cormack-Quinto spherical mean transform for spheres through the origin. The consideration extends to the corresponding dual transforms and  the relevant exterior/interior modifications.
The method relies on new results for the Gegenbauer-Chebyshev  integrals, which  generalize Abel type fractional integrals on the positive half-line.

 \end{abstract}

\maketitle

\section{Introduction}
\setcounter{equation}{0}

In the present article we solve open problems stated in  \cite{Ru14} and related to the structure of the kernels (null spaces) of several Radon-like transforms. These transforms are projectively equivalent to the hyperplane Radon transform of functions on $\rn$ and include
 the Funk transform on the sphere, its modification for  spherical slices through the  pole, the totally geodesic transform on the  hyperbolic space,  and the Cormack-Quinto  transform, which integrates functions on $\rn$ over  spheres passing through the origin.

All these transforms have  unilateral structure. This important fact allows us to neglect some singularities, which  restrict the corresponding classes of admissible functions. For example, in the case of the hyperplane Radon transform $Rf$ on $\rn$, we can exclude hyperplanes through the origin and  consider only   {\it almost all} hyperplanes in $\rn$ instead of {\it all} such planes. In such a setting, the behavior of $f$ at the origin becomes irrelevant, and therefore,
 our consideration covers not only the classical Radon transform, but also its exterior version, when both $f$ and $Rf$ are supported away from a fixed  ball.

For square integrable functions, the kernels of the exterior Radon transform  and the  interior spherical mean transform  were studied by Quinto \cite{Q82, Q83} who used the results of Ludwig \cite{Lud}; see also Cormack and   Quinto \cite{CoQ}. Our approach is different. It covers more general classes of functions and many other Radon-like transforms.  For example, we show that for any fixed $a\ge 0$, the kernel of the  Radon transform
  \[  (Rf) (\theta, t)=\intl_{\theta^{\perp}} f(t\theta +
u) \,d_\theta u, \quad   \theta \in \sn, \quad |t|>a,  \]
in the class of functions satisfying
\be \label {lkmuxAA}
\intl_{|x|>a_1}  \,\frac{|f(x)|} {|x|} \, dx<\infty \quad \forall \, a_1>a, \ee
is essentially  the set  of functions $\om_{j,m}(x)$ of the form
\[ \om_{j,m} (x)\!=\!|x|^{2j-m-n}Y_m(x/|x|); \qquad m\!=\!2,3, \ldots; \quad j\!=\!1,2,  \ldots, [m/2],\]
 where $[m/2]$ is the integer part of $m/2$, and $Y_m$ is a spherical harmonic of degree $m$. The fact that these functions are annihilated by the operator $R$ is not new; cf. \cite{Bom91, Bom09, Q82, Q83, Ru14}.  The crucial point and one of  the main results of the paper is that {\it $\om_{j,m} (x)$ exhaust the kernel of $R$ under the assumption} (\ref{lkmuxAA}); see Theorem  \ref{azw1a2} for the precise statement. Note that the assumption (\ref{lkmuxAA}) is pretty weak in the sense that it is necessary for the existence of the Radon transforms of radial functions; see Theorem  \ref {byvs1}  below.  Similar exact kernel descriptions are obtained for all Radon-like transforms mentioned above.

 The paper is organized as follows. Section 2 is devoted to Gegenbauer-Chebyshev fractional integrals that form a background of our approach. In Section 3 we  describe the kernels of the hyperplane Radon transform and its dual. The main results are stated in Theorems 3.7 and 3.6, which include the corresponding exterior and interior versions. Section 4 contains the description of the kernel of the Cormack-Quinto transform, which integrates functions on $\rn$ over spheres through the origin. Sections 5,6, and 7 contain similar results for the Funk transform on the sphere, the spherical slice transform, and  the totally geodesic transform on the  hyperbolic space, respectively. Our assumptions for functions are inherited from (\ref{lkmuxAA}) and provide the existence the corresponding Radon-like transforms in the Lebesgue sense.

 {\bf Notation.}
As usual,  $\bbz, \bbn, \bbr, \bbc$ denote the sets of all integers,
positive integers, real numbers, and complex numbers,
respectively; $\bbz_+ = \{ j \in \bbz: \ j \ge 0 \}$;  $ \bbr_+ = \{ a \in \bbr: \ a > 0 \}$.
We will be dealing with the following function spaces:\\
$C (\bbr_+)$ is the space of  continuous complex-valued functions on $\bbr_+$;
$ C_* (\bbr_+)=\{\vp \in C (\bbr_+):  \lim\limits_{t\to 0^+} \vp (t)<\infty; \; \sup\limits_{t> 0} t^k |\vp (t)|<\infty \;\forall k\in  \bbz_+\}$;
$ S(\bbr_+)=\{\vp \in C^\infty (\bbr_+):  \vp^{(\ell)}=(d/dt)^\ell \vp \in C_* (\bbr_+) \;\forall \ell\in  \bbz_+\}$;
$ C_c^\infty (\bbr_+) =\{\vp \in C^\infty (\bbr_+):  \supp \,\vp \; \text {\rm is compact and } \; 0\notin \supp \vp\}$.

In the following, $S^{n-1} = \{ x \in \bbr^n: \ |x| =1 \}$ is the unit sphere
in $\bbr^n= \bbr e_1 \oplus \cdots \oplus \bbr e_{n}$, where $e_1, \ldots,  e_{n}$ are the coordinate unit vectors.   For $\theta \in S^{n-1}$, $d\theta$ denotes  the surface element on $S^{n-1}$;  $\sigma_{n-1} =  2\pi^{n/2} \big/ \Gamma (n/2)$ is the surface area of $S^{n-1}$.  We set $d_*\theta= d\theta/\sigma_{n-1}$ for  the normalized surface element on $S^{n-1}$.

The letter $c$  denotes  an inessential positive constant that may vary at each  occurrence.

\section{Some Properties of the Gegenbauer-Chebyshev  Integrals}

  The  Gegenbauer polynomials $ C_m^\lam (t)$, $\lam > -1/2$, have the form
\be\label{kIIu}
   C_m^\lam (t)\!=\!\sum\limits_{j=0}^M c_{m,j} \, t^{m-2j}, \quad c_{m,j}\!= \!(-1)^j\,\frac{2^{m-2j}\,\Gam (m\!-\!j\!+\!\lam)}{\Gam (\lam)\, j!\,  (m\!-\!2j)!}, \ee
 where $M\!=\![m/2]$ is the integer part of $m/2$;  see \cite{Er}.  In the case  $\lam=0$, they are usually substituted by the Chebyshev polynomials $T_m (t)$.
If $|t|\le 1$, then
\be\label{kioxsru}  |C^\lam_m (t)|\le c\, \left\{\begin{array}{ll} 1,  & \mbox{if $m$ is even,}\\
|t|,  & \mbox{if $m$ is odd,} \qquad c\equiv c(\lam, m) =\const.\\
 \end{array}
\right.\ee
The same inequality holds for $T_m (t)$; cf. 10.9(18) and 10.11(22) in \cite{Er}.

\subsection{The right-sided  integrals}

The {\it right-sided Gegenbauer-Chebyshev integrals} of a function $f$ on $\bbr_+$ are defined by
\be \label {4gt6a}
(\G^{\lam, m}_{-} f)(t)=\frac{1}{c_{\lam,m}}\,\intl_t^\infty (r^2 - t^2)^{\lam -1/2}\, C^\lam_m \left (\frac{t}{r} \right )\, f (r)  \, r\, dr,\ee
\be \label {4gt6a1}  (\stackrel{*}{\G}\!{}_{\!-}^{\!\lam,m} f)(t)
=\frac{t}{c_{\lam,m}}\,\intl_t^\infty (r^2 - t^2)^{\lam -1/2}\, C^\lam_m \left (\frac{r}{t} \right )\, f (r)  \,\frac{dr}{r^{2\lam +1}},\ee
\be\label {gynko}
c_{\lam,m}=\frac{ \Gam (2\lam+m)\,\Gam (\lam+1/2)}{2 m!\, \Gam (2\lam)},  \qquad \lam >-1/2, \quad \lam \neq 0.\ee
 In the case  $\lam=0$, when the Gegenbauer polynomials are  substituted by the Chebyshev ones, we set
\be \label {4gt6at}
(\T_{-}^m f)(t)=\frac{2}{\sqrt{\pi}}\,\intl_t^\infty (r^2 - t^2)^{-1/2}\, T_m \left (\frac{t}{r} \right )\, f (r)  \, r\, dr,\ee
\be \label {4gt6a1t}  (\stackrel{*}{\T}\!{}_{\!-}^m f)(t)
=\frac{2t}{\sqrt{\pi}}\,\intl_t^\infty (r^2 - t^2)^{-1/2}\, T_m \left (\frac{r}{t} \right )\, f (r)  \,\frac{dr}{r}.\ee

In the following, all statements  are presented for the case of Gegenbauer polynomials  $C_m^\lam (t)$.  The corresponding statements for the Chebyshev polynomials can be formally obtained by setting $\lam=0$ and proved similarly.

\begin{proposition} \label{lo8xw} \cite[Proposition 3.1]{Ru14}
     Let $a>0$, $\lam >-1/2$. The integrals $(\G^{\lam, m}_{-} f)(t)$ and $(\stackrel{*}{\G}\!{}_{\!-}^{\!\lam,m} f)(t)$ are finite for almost all $t>a$ under the following conditions.

{\rm (i)} For $(\G^{\lam, m}_{-} f)(t)$:
\be\label{for10zg} \intl_a^\infty |f(t)|\, t^{2\lam-\eta }\, dt
<\infty,\qquad \eta=\left \{\begin{array} {ll} 0 &\mbox{if $m$ is even,}\\
1 &\mbox{if $m$ is odd.}\end{array}\right.\ee

{\rm (ii)} For $(\stackrel{*}{\G}\!{}_{\!-}^{\!\lam,m} f)(t)$:
\be\label{for10zgd} \intl_a^\infty |f(t)|\, t^{m-2}\, dt
<\infty.\ee
 \end{proposition}

 \begin{proposition} \label{lo8xw91} If $f\in S(\bbr_+)$, then $\G^{\lam, m}_{-}f$ is an infinitely differentiable function on $\bbr_+$ such that $t^{\gam -1} \G^{\lam, m}_{-}f \in L^1 (\bbr_+)$ for all $\gam>0$.
\end{proposition}
 \begin{proof} The infinite differentiability of $ (\G^{\lam, m}_{-}f)(t)$ is easily seen if we write
\[
(\G^{\lam, m}_{-}f)(t)=\frac {t^{2\lam +1} }{c_{\lam,m}}\,\intl_1^\infty (s^2 - 1)^{\lam -1/2}\, C^\lam_m \left (\frac{1}{s} \right )\, f (ts)  \, s\, ds.\]
  The second statement can be checked straightforward by making use of (\ref{kioxsru}).
\end{proof}

\begin{proposition} \label{lo8xw92} The operator $\stackrel{*}{\G}\!{}_{\!-}^{\!\lam,m}$ is injective on the class of continuous functions $f$ on $\bbr_+$ satisfying
\be\label{foNNgd}
t^{\gam -1} f(t) \in  L^1 (\bbr_+) \quad \text{for some}\quad\gam >m-1.\ee
\end{proposition}
 \begin{proof}  We write $\stackrel{*}{\G}\!{}_{\!-}^{\!\lam,m}f$ as a Mellin convolution
\[(\stackrel{*}{\G}\!{}_{\!-}^{\!\lam,m}f)(t)\!= \!\intl_0^\infty g(s) \,f\left (\frac{t}{s}\right )\, \frac{ds}{s},\quad g(s)\!=\!\frac{s}{c_{\lam,m}}\,(1\!-\!s^2)_+^{\lam -1/2}\, C^\lam_m \left (\frac{1}{s} \right ).\]
\[ \]
By the formula 2.21.2(25) from \cite{PBM2},
\bea
\tilde g(z)&=&\intl_0^\infty g(s) \,s^{z-1}\, ds\nonumber\\
&=& \frac{\displaystyle{\Gam \left(\frac{z\!+\!1\!-\!m}{2}\right)\,\Gam \left(\lam \!+\!\frac{z\!+\!1\!+\!m}{2}\right)}}{ \displaystyle{\Gam \left(\lam\!+\!\frac{z+1}{2}\right)\, \Gam \left(\lam\!+\!\frac{z\!+\!2}{2}\right)}}, \qquad Re\, z>m-1. \nonumber\eea
Hence, $f$ is uniquely reconstructed by the Mellin inversion formula (see, e.g., \cite{Mari})
\[
f(t)=\frac{1}{2\pi i}  \intl_{\gam -i\infty}^{\gam +i\infty}  \tilde f(z) \, \tilde g(z)\, dz, \qquad \gam >m-1,\]
which gives the desired injectivity result.
 \end{proof}

We will also need the Riemann-Liouville fractional integrals
\be \label {RLint}(I^\a_{-}f) (t) = \frac{1}{\Gamma (\alpha)} \intl_t^{\infty}
\frac{f(s) \,ds} {(s-t)^{1- \alpha}}, \qquad t>0, \quad \a>0.\ee
 The corresponding operators of fractional differentiation, which are defined as the left inverses of $I^\a_{-}$, will be denoted by $D^\a_{-}$.

\begin{proposition} \label{blo8xw} The operator $I^\a_{-}$ is a bijection of $S(\bbr_+)$ onto itself.
\end{proposition}
\begin{proof} This fact is not new, and the proof is given for the sake of completeness. Let $f\in S(\bbr_+)$,
\[\vp (t)=(I^\a_{-}f) (t) =\frac{1}{\Gamma (\alpha)} \intl_0^{\infty} s^{\alpha -1}\,
f(s+t)\, ds.\]
This function is infinitely differentiable for $t>0$ and all derivatives $\vp^{(\ell)} (t)=(I^\a_{-}f^{(\ell)}) (t)$ have finite limit as
$t\to 0^+$. Thus, $I^\a_{-}f \in  S(\bbr_+)$. The injectivity of $I^\a_{-}$ is a standard fact from Fractional Calculus; see, e.g.,
\cite{Ru96, SKM}, so that $D^\a_{-}  I^\a_{-}f =f$. Here  $D^\a_{-} $ has different forms, for instance, $(D^\a_{-} \vp)(t)= (-d/dt)^m (I^{m-\a}_{-} \vp)(t)$ for any integer $m>\a$.
Conversely, given any $\vp\in S(\bbr_+)$ and $\a>0$, for any integer $m>\a$ we have
\[\vp (t)= -\intl_t^{\infty} \vp' (s)\, ds= \cdots = (-1)^m (I^m_{-} f^{(m)})(t)=(I^\a_{-}\psi) (t),\]
where $\psi=(-1)^m I^{m-\a}_{-} f^{(m)}\in  S(\bbr_+)$ by the first part of the proof. Hence,  $I^\a_{-}:  S(\bbr_+) \to  S(\bbr_+)$ is surjective.
\end{proof}

\begin{proposition} \label{89n6g} \cite [Lemma 3.4]{Ru14}  Let $\lam >-1/2$, $m\ge 2$. Suppose that
\be\label{for10zgn} \intl_a^\infty |f(t)|\, t^{2\lam +m-1}\, dt<\infty\quad \forall \,a>0.\ee
 Then for almost all $t>0$,
\be\label{8vcmk}
(\stackrel{*}{\G}\!{}_{\!-}^{\!\lam,m}  \G^{\lam, m}_{-} f)(t)=2^{2\lam +1}(I_-^{2\lam+1} f)(t).\ee
\end{proposition}

\begin{definition} \label{DEF6g55} We denote by $\D_m (\bbr_+)$ the set of all functions $\vp\in C_c^\infty (\bbr_+)$ satisfying the moment conditions
\be\label{8vcmk11} \intl_0^\infty r^{m-2k} \, \vp (r)\, dr=0\quad \forall \;1\le k\le M, \quad M=[m/2].\ee
\end{definition}

\begin{proposition} \label{89n6g55} If $\vp \in \D_m (\bbr_+)$, then $\stackrel{*}{\G}\!{}_{\!-}^{\!\lam,m} \vp \in S(\bbr_+)$. Moreover, if
 $\supp \,\vp \subset [a,b]$, $0<a<b<\infty$, then $(\stackrel{*}{\G}\!{}_{\!-}^{\!\lam,m} \vp)(t)=0$ for all $t>b$.
\end{proposition}
\begin{proof} The second statement is an immediate consequence of the right-sided structure of $\stackrel{*}{\G}\!{}_{\!-}^{\!\lam,m} \vp$.  To prove the first statement, let $f= \stackrel{*}{\G}\!{}_{\!-}^{\!\lam,m} \vp$. Then
\[f(t)=\frac{1}{c_{\lam,m}}\,\intl_0^1 (1 - s^2)^{\lam -1/2}\, C^\lam_m  \left(\frac{1}{s}\right)\, \vp\left(\frac{t}{s}\right)  \,ds.\]
Since $\vp$ is compactly supported, this function is infinitely differentiable for $t>0$ and we can write
\[ f^{(\ell)}(t)=\frac{1}{2\,c_{\lam,m}}\,\intl_0^1 (1 - x)^{\lam -1/2}\, C^\lam_m  \left(\frac{1}{\sqrt {x}}\right)\, \vp^{(\ell)}\left(\frac{t}{\sqrt {x}}\right) \,\frac{dx}{(\sqrt {x})^{\ell +1}}.\]
By Taylor's formula, $(1 - x)^{\lam -1/2}=p_n(x)+c\, x^{n+1}\,\om_n (x)$, where   $p_n(x)$ is a polynomial of degree $n$, $c=\const$, and
\[
\om_n (x)=\intl_0^1 (1-yx)^{\lam-n-3/2} (1-y)^n dy.\]
 Hence (cf. (\ref{kIIu})),  for some $N>0$,
\[(1 - x)^{\lam -1/2}\, C^\lam_m  \left(\frac{1}{\sqrt {x}}\right)= \sum\limits _{j=0}^N c_j \,(\sqrt {x})^{2j-m}+c\,x^{n+1}\,\om_n (x)\,  C^\lam_m  \left(\frac{1}{\sqrt {x}}\right).\]
 Plugging this expression in $f^{(\ell)}$, we obtain $f^{(\ell)}=A+B$, where $A$ is a linear combination of integrals
 \[ A_{j,\ell} (t)= \intl_0^1 (\sqrt {x})^{2j-m-\ell -1} \vp^{(\ell)}\left(\frac{t}{\sqrt {x}}\right)\, dx=c_{j,\ell}\, t^{2j-m-\ell +1},\]
\[c_{j,\ell}= 2 \intl_t^\infty r^{m+\ell-2j-2}\vp^{(\ell)} (r)\, dr,\qquad j=0,1,\ldots , N; \quad \ell=0,1,\ldots ,\]
and
\bea
B&=&c\intl_0^1 \om_n (x)\, C^\lam_m  \left(\frac{1}{\sqrt {x}}\right)\, \vp^{(\ell)}\left(\frac{t}{\sqrt {x}}\right) \,x^{n+1-(\ell +1)/2}\, dx\nonumber\\
&=&2c\, t^{2n-\ell +3}\intl_t^\infty \om_n  \left( \frac{t^2}{r^2}\right)C^\lam_m  \left(\frac{r}{t}\right)\, \vp^{(\ell)} (r)\, r^{\ell -2n-4}\, dr.\nonumber\eea

If $\vp$ satisfies (\ref{8vcmk11}), then all  $A_{j,\ell}$ have   a finite limit as $t\to 0$. Furthermore, by (\ref{kIIu}), $B$ is a linear combination of the integrals
\[B_{j,\ell}= t^{2n-\ell +3+2j-m}\intl_t^\infty \om_n  \left( \frac{t^2}{r^2}\right)\, \vp^{(\ell)} (r)\, r^{\ell -2n-4+m-2j}\, dr.\]
Recall that $\supp \,\vp \!\subset \![a,b]$, $0\!<\!a\!<\!b\!<\!\infty$, and suppose that $0\!<\!t\!<\!a/2$. Then
\bea \om_n  \left( \frac{t^2}{r^2}\right)&=&\intl_0^1 \left(1-\frac{yt^2}{r^2}\right)^{\lam-n-3/2} (1-y)^n \,dy\nonumber\\
&\le&\intl_0^1 \frac{(1-y)^n \,dy}{(1-y/4)^{n+3/2-\lam}}=c<\infty.\nonumber\eea
It follows that $B_{j,\ell}=O(t^{2n-\ell +3+2j-m}) \to 0$  as $t\to 0$ if $n$ is big enough, which completes the proof.
\end{proof}

\begin{proposition} \label{89n6g66} Every function $\vp \!\in \!\D_m (\bbr_+)$ is represented as $\vp \!=\!\G^{\lam, m}_{-}\psi$, where
\be\label{ki9a} \psi= 2^{-2\lam -1} D_-^{2\lam +1} \stackrel{*}{\G}\!{}_{\!-}^{\!\lam,m} \vp \in S(\bbr_+).\ee If
 $\supp \,\vp \subset [a,b]$, $0<a<b<\infty$, then $\vp(t)=0$ for all $t>b$.
\end{proposition}
\begin{proof} By  Proposition \ref{89n6g55}, the function $f= \stackrel{*}{\G}\!{}_{\!-}^{\!\lam,m} \vp$ belongs to $S(\bbr_+)$ and equals zero for all $t>b$. Hence, by Proposition \ref{blo8xw}, the function $\psi= 2^{-2\lam -1} D_-^{2\lam +1}f$ also belongs to $S(\bbr_+)$ and equals zero for all $t>b$. To  show that $\vp =\G^{\lam, m}_{-}\psi$, let $F=\vp -\G^{\lam, m}_{-}\psi$. By Proposition \ref{lo8xw91},
$F\in C^\infty (\bbr_+)$  and
$t^{\gam -1} F \in L^1 (\bbr_+)$ for all $\gam>0$. Hence,  owing to Propositions \ref{89n6g} and \ref{lo8xw},
\[\stackrel{*}{\G}\!{}_{\!-}^{\!\lam,m} F=\stackrel{*}{\G}\!{}_{\!-}^{\!\lam,m} \vp -\stackrel{*}{\G}\!{}_{\!-}^{\!\lam,m}\G^{\lam, m}_{-}\psi=\stackrel{*}{\G}\!{}_{\!-}^{\!\lam,m} \vp- I_-^{2\lam +1} D_-^{2\lam +1}\stackrel{*}{\G}\!{}_{\!-}^{\!\lam,m} \vp=0.\]
Now, the required result follows from the injectivity of the operator $\stackrel{*}{\G}\!{}_{\!-}^{\!\lam,m}$;  see Proposition \ref{lo8xw92}.
\end{proof}

\subsection{The left-sided  integrals}

Let $\lam >-1/2$,  $m \in \bbz_+$, $0<a\le \infty$. The {\it  left-sided Gegenbauer and Chebyshev fractional integrals} on the interval $(0,a)$ are defined as follows.  For  $\lam \neq 0$, we set
\be \label {4gt6ale}
(\G^{\lam, m}_{+} f)(r)=\frac{r^{-2\lam}}{c_{\lam,m}}\,\intl_0^r (r^2 - t^2)^{\lam -1/2}\, C^\lam_m \left (\frac{t}{r} \right )\, f (t)  \, dt,\ee
\be \label {4gt6a1le}  (\stackrel{*}{\G}\!{}_{\!+}^{\!\lam,m} f)(r)
=\frac{1}{c_{\lam,m}}\,\intl_0^r (r^2 - t^2)^{\lam -1/2}\, C^\lam_m \left (\frac{r}{t} \right )\, f (t) \, t\,dt,\ee
$c_{\lam,m}$ being defined by (\ref{gynko}), $0<r<a$. In the case $\lam = 0$ we denote
\be \label {4gt6atle}
(\T_{+}^m f)(r)=\frac{2}{\sqrt{\pi}}\,\intl_0^r (r^2 - t^2)^{-1/2}\, T_m \left (\frac{t}{r} \right )\, f (t)  \,  dt,\ee
\be \label {4gt6a1tle}  (\stackrel{*}{\T}\!{}_{\!+}^m f)(r)
=\frac{2}{\sqrt{\pi}}\,\intl_0^r (r^2 - t^2)^{-1/2}\, T_m \left (\frac{r}{t} \right )\, f (t) \, t\,  dt.\ee
The left-sided  integrals are expressed through the right-sided ones by the formulas
\bea &&\label {4gt6a1ldr}
(\G^{\lam, m}_{+} f)(r)=\frac{1}{r}\, (\G^{\lam, m}_{-} f_1)\left(\frac{1}{r}\right), \qquad f_1(t)=\frac{1}{t^{2\lam +2}} \, f\left(\frac{1}{t}\right);\qquad \\
&& \label {4gt6a1ldz}
(\stackrel{*}{\G}\!{}_{\!+}^{\!\lam,m} f)(r)=r^{2\lam}(\stackrel{*}{\G}\!{}_{\!-}^{\!\lam,m} f_2)\left(\frac{1}{r}\right), \qquad f_2(t)=\frac{1}{t} \, f\left(\frac{1}{t}\right).\eea

These formulas combined with Proposition \ref{lo8xw} give the following statement.

 \begin{proposition} \label{lo8xwADD}  Let  $a>0$,  $\lam >-1/2$.  The integrals (\ref{4gt6ale})-(\ref{4gt6a1tle})
 are absolutely convergent  for almost all $r<a$ under the following conditions.

{\rm (i)} For (\ref{4gt6ale}), (\ref{4gt6atle}):
 \be\label {026g094t}   \intl_0^a t^\eta |f(t)|\, dt <\infty,\qquad \eta=\left \{\begin{array} {ll} 0 &\mbox{if $m$ is even,}\\
1 &\mbox{if $m$ is odd.}\end{array}\right.\ee

{\rm (ii)} For (\ref{4gt6a1le}), (\ref{4gt6a1tle}):
 \be\label {026gt}
  \intl_0^a  t^{1-m} |f(t)|\, dt <\infty.\ee
 \end{proposition}

 \begin {lemma} \label {marti} \cite[Proposition 3.7]{Ru14}
If $m=0,1$, then $\G^{\lam, m}_{+}$ is injective on $\bbr_+$  in the class of functions satisfying (\ref{026g094t}) for all $a>0$.
 If $m\ge 2$, then $\G^{\lam, m}_{+}$ is non-injective in this class of functions.
Specifically, let  $f_k (t)=t^k$,  where $k$ is a nonnegative  integer such that $ m-k=2,4,  \ldots$.
 Then $(\G^{\lam, m}_{+} f_k)(t)=0$ for all $t>0$.
\end {lemma}

An important question is: Are there any other functions in the kernel of the operator  $\G^{\lam, m}_{+}$ rather than $f_k$ in Lemma \ref{marti}?  Below we give a negative answer to this question under certain conditions, which are very close to (\ref{026g094t}).
Let
\be\label {PLUH} \varkappa_{\lam, m} (t)=\left \{\begin{array} {ll} 1 \; &\text{\rm if $m$ is even,}\\
t^{1+2\lam}\; &\text{\rm if $m$ is odd,  $-1/2<\lam<0$,}\\
t\,(1+ | \log t|) \; &\text{\rm if $m$ is odd,  $\lam=0$,}\\
t \; &\text{\rm if $m$ is odd,  $\lam>0$}. \end{array}\right.\ee
Given $0<a\le \infty$, we denote by $L^1_\varkappa (0,a)$ the set of all function $f$ on  $(0,a)$ such that
\be\label {PLUH1} \intl_0^{a_1} |f(t)|\, \varkappa_{\lam, m} (t)\, dt \le \infty \quad \forall\, a_1<a.\ee
Clearly,  $L^1_{loc} (0,a)\subset L^1_\varkappa (0,a)$.

 \begin {lemma} \label {marti2} Suppose that $m\ge 2$, $M=[m/2]$, $0<a\le\infty$. If $f\in L^1_\varkappa (0,a)$  and $(\G^{\lam, m}_{+} f)(r)=0$ for almost all $0<r<a$, then
\be\label{ch2uha19X} f(t)=\sum\limits_{j=1}^{M} c_{j}\, t^{m-2j} \quad \mbox{a.e. on $(0,a)$} \ee
with some coefficients $c_j$.
\end{lemma}
\begin{proof}
We introduce auxiliary functions \[\vp_{m-2k}\in C_c^\infty (0,a), \qquad k\in \{1,2,\ldots, M\}, \] satisfying
\be\label{ch2uha191}
\intl_0^a t^{m-2j} \vp_{m-2k} (t)\, dt=\del_{j,k}=\left \{\begin{array} {ll} 1\; &\text{\rm if} \; j=k, \\
0 \; &\text{\rm otherwise.}\\
\end{array}\right.\ee
The existence of such functions is a consequence of the general fact from functional analysis for bi-orthogonal systems; see, e.g., \cite [p. 160]{Kak}. We set
\be\label{lko2}
 c_{j}=\intl_0^a f(t)\, \vp_{m-2j} (t)\, dt.\ee
Let $\om \in C_c^\infty (0,a)$ be an arbitrary test function, and let
\be\label{lko} \vp (t)=\om (t)-\sum\limits_{j=1}^{M} \vp_{m-2j} (t)\intl_0^a s^{m-2j} \om (s)\, ds \quad (\in C_c^\infty (0,a) ).\ee
Then for any $k\in \{1,2,\ldots, M\}$,
\bea
\intl_0^a t^{m-2k} \vp (t)\, dt&=&\intl_0^a t^{m-2k} \om (t)\, dt\nonumber\\
&-&\sum\limits_{j=1}^{M} \Bigg [\intl_0^a t^{m-2k}  \vp_{m-2j} (t)\, dt\Bigg ] \Bigg [\intl_0^a s^{m-2j} \om (s)\, ds\Bigg ].\nonumber\eea
By (\ref{ch2uha191}), this gives
\[\intl_0^a t^{m-2k}  \vp (t)\, dt=0\quad \forall\, k\in \{1,2,\ldots, M\}.\]
Suppose that $\vp (t)\equiv 0$ on some interval $(a_1, a)$ with $a_1<a$, and define  $\tilde \vp \in C_c^\infty(\bbr_+)$ so that
$\tilde \vp (t)=\vp (t)$ if $t\le a_1$ and   $\tilde \vp (t)=0$ if $t>a_1$.
By Proposition \ref {89n6g66},  $\tilde \vp$ is represented as $\tilde \vp \!=\!\G^{\lam, m}_{-}\psi$, where $\psi$ belongs to $ S(\bbr_+)$ and equals zero on $(a_1,\infty)$. Then
\bea \intl_0^\infty f(t) \tilde \vp (t)\, dt&=&\intl_0^\infty f(t)\, (\G^{\lam, m}_{-}\psi)(t)\, dt\nonumber\\
\label{lko1}&=& \intl_0^{a_1} \psi (r) (\G^{\lam, m}_{+} f)(r)\, r^{2\lam +1}\, dr=0.\eea
Hence, by (\ref{lko}), (\ref{lko1}), and  (\ref{lko2}),
\bea
&&\intl_0^a f(t) \,\om (t)\, dt=\intl_0^\infty f(t) \left [\vp (t)+\sum\limits_{j=1}^{M} \vp_{m-2j} (t) \intl_0^a s^{m-2j} \om (s)\, ds \right ]\, dt\nonumber\\
&&=\sum\limits_{j=1}^{M} c_{j} \intl_0^a  s^{m-2j} \om (s)\, ds=\intl_0^a  \left [\sum\limits_{j=1}^{M}  c_{j} s^{m-2j}\right ]\,  \om (s)\, ds.\nonumber\eea
This gives (\ref{ch2uha19X}).

To complete the proof, we must justify application of Fubini's theorem in (\ref{lko1}). Replacing $(\G^{\lam, m}_{+} f)(r)$ according to (\ref{4gt6ale}), and taking absolute values, we have
\bea
I&\equiv&\intl_0^{a_1} |\psi (r)| \,r\, dr\intl_0^r    (r^2 - t^2)^{\lam -1/2}\, \left| C^\lam_m \left (\frac{t}{r} \right )\right |\, |f (t)|  \, dt\nonumber\\
&\le& c\, \intl_0^{a_1}  \,r\, dr\intl_0^r    (r^2 - t^2)^{\lam -1/2}\,  \left (\frac{t}{r} \right )^\eta\, |f (t)|  \, dt,
\nonumber\eea
where $\eta=0$ if $m$ is even and $\eta=1$ if $m$ is odd.
If $\eta=0$, then
\[
I\le c\,\intl_0^{a_1} |f (t)|  \, dt \intl_t^{a_1}   (r^2 - t^2)^{\lam -1/2}\,\,r\, dr\le c_1\, \intl_0^{a_1} |f (t)|  \, dt<\infty.\]
If $\eta=1$, then
\[
I\le c\,\intl_0^{a_1} |f (t)| \,t \,g(t)\,  dt,\]
where
\[g(t)=\intl_t^{a_1}   (r^2 - t^2)^{\lam -1/2}\, dr=t^{2\lam}\intl_1^{a_1/t}   (s^2 - 1)^{\lam -1/2}\, ds.\]
The behavior of $g(t)$ as $t\to 0$ can be easily examined by considering the cases indicated in (\ref{PLUH}). This  gives
\[
I\le c\,\intl_0^{a_1} |f(t)|\, \varkappa_{\lam, m} (t)\, dt < \infty.\]
\end{proof}

Lemma \ref{marti2} combined with  (\ref{4gt6a1ldr}) gives the following result for the right-sided Gegenbauer-Chebyshev integrals.

Let
\be\label {rightG} \tilde\varkappa_{\lam, m} (t)=\left \{\begin{array} {ll} t^{2\lam} \; &\text{\rm if $m$ is even,}\\
t^{-1}\; &\text{\rm if $m$ is odd,  $-1/2<\lam<0$,}\\
t^{2\lam-1}\,(1+ | \log t|) \; &\text{\rm if $m$ is odd,  $\lam=0$,}\\
t^{2\lam-1} \; &\text{\rm if $m$ is odd,  $\lam>0$}. \end{array}\right.\ee

Given $a\ge \infty$, we denote by $L^1_{\tilde\varkappa} (a, \infty)$ the set of all function $f$ on  $(a, \infty)$ such that
\be\label {rightG1} \intl^\infty_{a_1} |f(t)|\, \tilde\varkappa_{\lam, m} (t)\, dt \le \infty \quad \forall\, a_1>a.\ee

\begin {lemma} \label {marti3} Suppose that $m\ge 2$, $M=[m/2]$, $a\ge 0$. If $f\in L^1_{\tilde\varkappa} (a, \infty)$  and $(\G^{\lam, m}_{-} f)(r)=0$ for almost all $r>a$, then
\be\label{ch2uha19} f(t)=\sum\limits_{k=0}^{M-1} c_{k} t^{2k-m-2\lam} \quad \mbox{a.e. on $(a,\infty)$} \ee
with some coefficients $c_k$.
\end{lemma}

Lemma \ref{marti3} is an important complement of the following statement, which was proved in \cite[Lemma 3.3]{Ru14}.

\begin {lemma} \label{89srg}  Let $\lam >-1/2$. If $m=0,1$, then $\G^{\lam, m}_{-}$ is injective on $\bbr_+$ in the class of functions satisfying (\ref{for10zg}) for all $a>0$.
 If $m\ge 2$, then $\G^{\lam, m}_{-}$ is non-injective in this class of functions.
Specifically, let  $f_k (t)=t^{-2\lam-k-2}$,  where $k$ is a nonnegative  integer such that $ m-k=2,4,  \ldots$.
 Then $(\G^{\lam, m}_{-} f_k)(t)=0$ for all $t>0$.
\end {lemma}

According to Lemma \ref{marti3}, the functions $f_k$  exhaust the kernel of the operator  $\G^{\lam, m}_{-}$ in the  space $L^1_{\tilde\varkappa} (a, \infty)$, $a\ge 0$.

\section{Radon Transforms on $\rn$}
 We recall some known facts; see, e.g.,  \cite{GGG2,  H11,  Ru13b, Ru15}.
Let $\Pi_n$ be the set of all unoriented hyperplanes in
$\bbr^n$.
The Radon transform of a function $f$ on $\rn$
  is defined by the formula
\be\label{rtra} (Rf)(\xi)= \intl_{\xi} f(x) \,d_\xi x,\qquad \xi \in
\Pi_n,\ee
provided that this integral exists. Here $d_\xi x$ denotes  the Euclidean volume element  in $\xi$.
Every hyperplane $\xi\in \Pi_n$ has the form $\xi=\{x:\, x\cdot \th =t\}$, where $\theta \in \sn$, $t\in \bbr$.
Thus, we can write (\ref{rtra}) as
  \be\label{rtra1}  (Rf) (\theta, t)=\intl_{\theta^{\perp}} f(t\theta +
u) \,d_\theta u,\ee where
$\theta^\perp=\{x: \, x \cdot \theta=0\}$ is the hyperplane
orthogonal to $\theta$ and passing through the origin,
$d_\theta u$ is the Euclidean volume element  in $\theta^{\perp}$. We set $Z_n=S^{n-1}\times \bbr$ and equip $Z_n$ with the product measure
$d_*\theta dt$, where $d_*\theta=\sig_{n-1}^{-1} d\theta$ is the normalized surface measure on $\sn$.

Clearly, $ (Rf) (\theta,t) = (Rf)(-\theta, -t)$ for every $(\theta,t) \in Z_n$.

\begin{theorem}\label{byvs1} {\rm (cf. \cite[Theorem 3.2]{Ru13b})} \   If
\be \label {lkmux}
\intl_{|x|>a}  \,\frac{|f(x)|} {|x|} \, dx<\infty \quad \forall \, a>0, \ee
 then  $(Rf)(\xi)$ is finite for  almost all $\xi \in \Pi_n$.
If $f$ is nonnegative, radial, and (\ref{lkmux}) fails, then  $(Rf)(\xi)\equiv \infty$.
\end{theorem}

The    dual Radon transform is an averaging operator that takes a  function $\vp (\theta,t)$ on $Z_n$  to a function  $(R^*\vp)(x)$ on  $\bbr^n$  by the formula
\be\label{durt}
(R^*\vp)(x)= \intl_{S^{n-1}} \vp(\theta, x\cdot \theta)\,d_*\theta. \ee
The operators $R$ and $R^*$ can be expressed one through another.

\begin{lemma}\label {iozesf} \cite[Lemma 2.6] {Ru14} Let $x\neq 0$, $t\neq 0$,
\be \label {jikbVF} (A\vp)(x)\!=\! \frac{1}{|x|^n} \,\vp\left (\frac{x}{|x|}, \frac{1}{|x|}\right ), \qquad (Bf)(\theta, t)\!=\!\frac{1}{|t|^n} \, f \left (\frac{\theta}{t}\right ).\ee
The following equalities hold provided that the expressions on either side exist in the Lebesgue sense:
\be \label {jikb}(R^* \vp)(x)\!=\!\frac{2}{|x|\, \sig_{n-1}}\, (RA\vp)\left (\frac{x}{|x|}, \frac{1}{|x|}\right ), \ee
\be\label {jikbBU} (Rf)(\theta, t)\!=\!\frac{\sig_{n-1}}{2|t|}\, (R^* Bf)\left (\frac{\theta}{t}\right ).\ee
\end{lemma}

Theorem \ref{byvs1} combined with (\ref{jikb}) gives the following
\begin{corollary} \label{gttgzuuuuh} If $\vp (\th,t)$ is  locally integrable on $Z_n$, then the dual Radon transform $(R^*\vp)(x)$ is finite for  almost all $x\in\rn$. If $\vp (\th,t)$ is nonnegative, independent of $\th$, i.e.,  $\vp (\th,t)\equiv \vp_0 (t)$, and such that
\[\intl_0^a \vp_0 (t)\, dt =\infty,\]
for some $a>0$, then  $(R^*\vp)(x)\equiv \infty$.
 \end{corollary}

We fix a real-valued orthonormal basis $\{Y_{m,\mu}\}$ of spherical harmonics in $L^2(S^{n-1})$; see, e.g., \cite{Mu}. Here
 $m\in \bbz_+$ and $ \mu=1,2, \ldots d_n (m)$, where
\be\label{kWSQRT} d_n(m) =(n+2m-2)\,
\frac{(n+m-3)!}{m! \, (n-2)!}\ee
is the dimension of the subspace of spherical harmonics of degree $m$.

\begin{lemma} \label{ppo9q} \cite[Lemma 4.3] {Ru14}  Let  $\lam=(n-2)/2$, $\vp(\theta,t)=v(t)\,Y_m (\theta)$, where  $Y_m$ is a spherical harmonic of degree $m$ and $v(t)$ is a locally integrable function on $\bbr$ satisfying $v(- t)=(-1)^m v(t)$.  Then $(R^*\vp)(x)\equiv (R^*\vp)(r\theta)$ is finite for all $\theta\in \sn$ and almost all $r>0$. Furthermore,
\be\label {poxe190} (R^*\vp)(r\theta)=u(r)\,Y_m (\th).\ee The function  $u(r)$    is represented by the Gegenbauer  integral (\ref{4gt6ale}) (or the  Chebyshev integral (\ref{4gt6atle})) as follows.

\noindent For $n\ge 3:$
\be\label {885frq} u(r)=\frac{r^{-2\lam}}{\tilde  c_{\lam,m}}\,\intl_0^r (r^2 - t^2)^{\lam -1/2}\, C^\lam_m \left (\frac{t}{r} \right )\, v (t)  \,  dt=   \pi^{\lam +1/2}  (\G^{\lam, m}_{+} v)(t),\ee
\[
\tilde c_{\lam,m}=\frac{\pi^{1/2} \Gam (2\lam+m)\,\Gam (\lam+1/2)}{2 m!\, \Gam (2\lam)\, \Gam (\lam+1)}.\]

\noindent For $n=2:$
\be \label {4gt6atq}
u(r)=\frac{2}{\pi}\intl_0^r (r^2 - t^2)^{-1/2}\, T_m \left (\frac{t}{r} \right )\,  v (t)  \,  dt= \pi^{1/2} (\T_{+}^m v)(t).\ee
\end{lemma}

In parallel with the Radon transforms $Rf$ and  $R^*\vp$ defined above, we shall also consider their exterior and interior versions, respectively. For  $a>0$, we denote
\[ B_a^+ =\{ x\in \bbr^n: |x|<a\}, \qquad   B_a^- =\{ x\in \bbr^n: |x|>a\}, \]
\[  C_a^+ =\{ (\th, t)\in Z_n: |t|<a\}, \qquad C_a^- =\{ (\th, t)\in Z_n: |t|>a\}.\]
When dealing with the {\it exterior Radon transform}, we assume that $f$ is defined  on  $B_a^-$ and $(Rf)(\th, t)$ is considered  for $(\th, t)\in C_a^-$. Similarly, in the study of the {\it interior dual Radon transform}, it is supposed that $\vp$ is defined on $C_a^+$ and the values of $(R^*\vp)(x)$ lie in  $B_a^+$.

Lemma \ref{ppo9q} implies the following statement; cf. \cite[Lemma 4.7]{Ru14}.
\begin{lemma} \label{zaehle} Let $0<a\le \infty$. If $\vp \in L^1_{loc}(C_a^+)$ is  even, then for almost all $r\in (0,a)$,
\be \label{zaehl2}
(R^*\vp)_{m,\mu}(r)\!\equiv\!\intl_{S^{n-1}}\!\!(R^*\vp) (r\th)\, Y_{m,\mu} (\th)\, d\th\!= \!\pi^{\lam +1/2}\,(\G^{\lam, m}_{+} \vp_{m,\mu})(r),\ee
 where $\lam=(n-2)/2$ and $\G^{\lam, m}_{+} \vp_{m,\mu}$ is the  Gegenbauer  integral  (\ref{4gt6ale})  (or the  Chebyshev integral (\ref{4gt6atle})).
 \end{lemma}

The next  theorem gives the description of the kernel of $R^*$  in terms of the  Fourier-Laplace coefficients
\be \label{zaehl2YU}  \vp_{m,\mu} (t)\!=\!\intl_{S^{n-1}}\vp (\th, t)\, Y_{m,\mu} (\th)\, d\th. \ee
\begin{theorem} \label{zaeh} Let $\vp (\th, t)$ be an even  locally integrable  function on  $C_a^+$, $0<a\le \infty$.

{\rm (i)} Suppose that  $\vp_{m,\mu} (t)=0$ for almost all $t\in (0,a)$ if $m=0,1$, and $\vp_{m,\mu}$ is a linear combination of the form
\be\label{KOKLIp} \vp_{m,\mu} (t)=\sum\limits_{j=1}^{M} c_{j}\, t^{m-2j}, \qquad c_j=\const, \quad M=[m/2],\ee
if $m\ge 2$. Then $(R^*\vp)(x)=0$ for almost all $x\in B_a^+$.

{\rm (ii)}  Conversely, if  $(R^*\vp)(x)=0$ for almost all $x\in B_a^+$, then,  for all $\mu=1,2, \ldots, d_n (m)$ and almost all $t<a$, the following statements hold.

\noindent  (a) If $m=0,1$, then $\vp_{m,\mu} (t)= 0$.

\noindent  (b)  If $m\ge 2$,  then  $\vp_{m,\mu} (t)$ has the form (\ref{KOKLIp})  with some constants $c_j$.
\end{theorem}
\begin{proof} The statement (i) was proved in \cite[Theorem 4.5]{Ru14}. To prove (ii), we observe that if  $(R^*\vp)(x)=0$ for almost all $x\in B_a^+$, then, by (\ref{zaehl2}),  $(\G^{\lam, m}_{+} \vp_{m,\mu})(r)=0$ for almost all $r\in (0,a)$ and all $m,\mu$. Because $\vp\in L^1_{loc} (C_a^+)$, then
$\vp_{m,\mu}\in L^1_{loc} (0,a)$, and (\ref{KOKLIp}) is an immediate consequence of Lemmas  \ref{marti} and \ref{marti2}.
\end{proof}

Theorem \ref{zaeh} together with Lemma \ref{iozesf} imply the following result for the Radon transform $Rf$.

\begin{theorem} \label {azw1a2} Given $0\le a< \infty$,  suppose that
\be\label{azw1X}  \intl_{|x|>a_1} \!\frac{|f(x)|}{|x|}\, dx<\infty\quad \mbox{for  all $a_1>a$},\ee
and let $(Rf)(\theta, t)\!=\!0$  almost everywhere on $C_a^-$. Then, for almost all $r>a$,  all the Fourier-Laplace coefficients
\[f_{m,\mu} (r)=\intl_{S^{n-1}}f (r\th)\, Y_{m,\mu} (\th)\, d\th\]
have the form
\be\label{azw2} f_{m,\mu} (r)=\left\{\begin{array} {ll} 0 \!&if\quad m\!=\!0,1,\\
\sum\limits_{j=1}^{[m/2]} c_{j}\, r^{2j-m-n} &if\quad m\!\ge \!2,
\end{array} \right.\ee
 with  some constants $c_j$. Conversely, for any constants $c_j$ and any  $f$ satisfying (\ref{azw1X}) and (\ref{azw2}), we have $(Rf)(\theta, t)=0$  almost everywhere on $C_a^-$.
\end{theorem}

\begin{proof} Suppose $a>0$ (if $a=0$ the changes are obvious). Note that  (\ref{azw1X}) implies $(Bf)(\theta, t)\equiv |t|^{-n} f(\th/t)\in L^1_{loc } (C_{1/a}^+)$.
If  $Rf=0$  a.e. on $C_a^-$, then $R^* Bf=0$ a.e. on $B_{1/a}^+$. Hence, by  Theorem \ref{zaeh},  for $t\in (0,1/a)$ we have
\[(Bf)_{m,\mu} (t)\!=\!t^{-n} \!\!\intl_{S^{n-1}} \!\!f \left (\frac{\theta}{t}\right )\, Y_{m,\mu} (\th)\, d\th \!=\!\left \{ \begin{array} {ll} 0  \quad &{\rm if} \!\quad m=0,1,\\
\sum\limits_{j=1}^{M} c_{j}\, t^{m-2j} &{\rm if}\! \quad m\ge2.
\end{array}\right.\]
 Changing variable $t=1/r$, we obtain (\ref{azw2}). Conversely, if  $f_{m,\mu} (r)=0$ for  $m=0,1$, and (\ref{azw2}) holds  for  $m\ge 2$, then $(Bf)_{m,\mu} (t)=0$ if $ m=0,1$, and
 $(Bf)_{m,\mu} (t)=\sum\limits_{j=1}^{M} c_{j}\, t^{m-2j}$ if $ m\ge2$.
The last equality  is obvious for $t>0$. If $t<0$, then
\bea (Bf)_{m,\mu} (t)&=&\intl_{S^{n-1}} (Bf)(\theta, t)\, Y_{m,\mu} (\th)\, d\th\nonumber\\
&=&(-1)^m \intl_{S^{n-1}} (Bf)(\theta, |t|)\, Y_{m,\mu} (\th)\, d\th\nonumber\\
&=&(-1)^m \sum\limits_{j=1}^{M} c_{j}\, |t|^{m-2j}=\sum\limits_{j=1}^{M} c_{j}\, t^{m-2j}.\nonumber\eea
 Hence, by Theorem \ref{zaeh},  $R^* Bf=0$ a.e. on $B_{1/a}^+$ and therefore, by (\ref{jikbBU}), $Rf=0$ a.e. on $C_a^-$.
\end{proof}

\begin{remark} \label{remazw2}{\rm An interesting open problem is the following: {\it What would be the structure of the kernel   of the Radon transform $R$  if  the action of $R$ is considered on  functions, which may not  satisfy} (\ref {azw1X}) ?
}
\end{remark}

This question is intimately connected with the remarkable result of Armitage and Goldstein \cite{ArG} who proved that there is a nonconstant harmonic function $h$ on $\rn$, $n\ge 2$, such that $\int_\xi |h|<\infty$ and $\int_\xi  h=0$ for every  hyperplane $\xi$; see also Zalcman \cite{Za82} ($n=2$) and Armitage \cite{Arm}. One can easily show that $h$ does not  satisfy (\ref {azw1X}). Indeed, suppose the contrary, and let $\{h_{m,\mu} (r)\}$ be the set of all Fourier-Laplace coefficients of  $h(x)\equiv h(r\th)$.   Then for any $a>0$ and $m\ge 2$,
\bea &&\intl_0^a |h_{m,\mu} (r)|\, dr\le\intl_0^a dr \intl_{\sn} |h(r\th) Y_{m,\mu} (\th)|\, d\th\nonumber\\
&&=\intl_{|x|<a} |h(x) Y_{m,\mu} (x/|x|)|\, \frac{dx}{|x|^{n-1}}\le c\intl_{|x|<a} \frac{dx}{|x|^{n-1}}<\infty.\nonumber\eea
On the other hand, the inequality
\[\intl_0^a \left |\sum\limits_{j=1}^{M} c_{j}\, r^{2j-m-n} \right |\, dr<\infty\]
is possible only if all $c_{j}$ are zeros. The latter means that $h$ consists only of harmonics of degree $0$ and $1$. However, by Theorem \ref{azw1a2},   these harmonics must be zero. Hence, $h(x)\equiv 0$, which contradicts the Armitage-Goldstein's result.

\section{The Cormack-Quinto Transform}\label{Cormack-Quinto}

The  Cormack-Quinto transform
\be\label {Corma} (\Q f) (x)=\intl_{\sn} f (x+|x|\,\th)\, d_*\th \ee
assigns to a function $f$ on  $\rn$ the mean values of $f$ over  spheres passing to the origin.   In (\ref{Corma}), $f$ is integrated   over the sphere of radius $|x|$ with center at $x$.

There is a remarkable connection between (\ref{Corma}) and the dual Radon
transform (\ref{durt}).
\begin{lemma} \label {CormaQ} \cite {CoQ, Ru14}  Let $n\ge 2$. Then
\be\label {Corma2}
(\Q f)  (x)=  |x|^{2-n}(R^*\vp)(x), \qquad \vp (\th, t)=(2|t|)^{n-2} \,
  f(2t\th),\ee
provided that  either side of this equality exists in the Lebesgue sense.
\end{lemma}

A consequence of Lemma \ref{CormaQ} is the following description of the kernel of $\Q$ inherited from Theorem  \ref{zaeh}.
We recall that the Fourier-Laplace coefficients of $f$ are defined by
\be \label{zaehl1Qui}  f_{m,\mu} (r)=\intl_{S^{n-1}}f (r\th)\, Y_{m,\mu} (\th)\, d\th, \qquad r>0. \ee

\begin{theorem} \label {zaehQuin} Given $0<a\le \infty$,  suppose that
\be\label {Corma3f} \intl_{|x|<a_1} \frac{|f  (x)|}{|x|}\, dx <\infty\quad \forall \, a_1<2a,\ee
and let $(\Q f)(x)=0$ for almost all $x\in B_a^+$. Then  all the Fourier-Laplace coefficients (\ref{zaehl1Qui})
have the form
\be\label{azw2INTROQ} f_{m,\mu} (r)=\left\{\begin{array} {ll} 0 \!&if\quad m\!=\!0,1,\\
\sum\limits_{j=1}^{[m/2]} c_{j}\, r^{m-2j} &if\quad m\!\ge \!2,
\end{array} \right.\ee
for almost all $r\in (0,2a)$ with  some constants $c_j$. Conversely, for any constants $c_j$ and any  $f$ satisfying (\ref{Corma3f}) and (\ref{azw2INTROQ}), we have $(\Q f)(x)=0$ a.e. on $B_a^+$.
\end{theorem}

\section{The Funk Transform}\label{Connection}

 The Funk transform $F$ assigns to a function $f$ on a sphere the integrals of $f$ over ``equators''. Specifically, for the  unit sphere $S^n$ in $\bbr^{n+1}$ we have
 \be\label{Funk.aa}
 (Ff)(\th)=\!\!\!
 \intl_{\{  \sig\in S^{n} : \,
\th \cdot \sig =0\}} \!\!\!\! f(\sig) \,d_\th \sig,
\ee
where $d_\th \sig$ stands for the $O(n+1)$-invariant probability measure on the $(n-1)$-dimensional section $\{\sig \in S^{n}:
\th \cdot \sig =0\}$; see, e.g., \cite{GGG2, H11, Ru15}.

Let $e_1, \ldots, e_{n+1}$ be the coordinate unit vectors, $\rn\!=\! \bbr e_1 \oplus \cdots \oplus \bbr e_{n}$,
\be\label{r hemisphere}   S^{n}_+=\{ \th=(\th_1, \ldots, \th_{n+1})\in S^{n}: 0<\th_{n+1} \le 1\}.\ee
It is known that $\ker F=\{0\}$ if the action of $F$ is considered on even integrable functions. Below we prove that the structure of the kernel of $F$ is different  if the functions under consideration  allow  non-integrable singularities at the poles $\pm e_{n+1}$, so that the Funk transform still exists in the a.e. sense.

 Consider the projection map
 \be\label{Con22on} \rn \ni x \xrightarrow{\;\mu\;} \th\in S^{n}_+, \qquad \th=\mu (x)=\frac{x+e_{n+1}}{|x+e_{n+1}|}.\ee
This  map  extends to the bijection $\tilde \mu$ from the set $\Pi_n$ of all unoriented hyperplanes in $\rn$ onto the set
    \be\label{Con22on5} \tilde S^{n}_+=\{\th =(\th_1, \ldots, \th_{n+1})\in S^{n}: \,  0\le \th_{n+1}<1\}.\ee
cf. (\ref{r hemisphere}). Specifically, if $\t =\{x\in \rn: x\cdot \eta =t\}\in \Pi_n$, $\eta\in \sn \subset \rn$, $t\ge 0$, and $\tilde \t$ is the $n$-dimensional subspace containing the lifted plane $\t +e_{n+1}$, then $\th \in \tilde S^{n}_+$ is defined to be a normal vector to $\tilde \t$.

The above notation is  used in the following theorem.

\begin{theorem} \label{Con22on22} \cite[Theorem 6.1]{Ru14} Let $g(x) = (1+|x|^2)^{-n/2} f(\mu (x))$, $x\in \rn$, where $f$ is an even function  on $S^{n}$.  The Funk transform $F$  and the  Radon transform $R$ are related by the formula
\be\label{Con22on2} (Ff)(\th)= \frac{2}{\sig_{n-1}\, \sin d (\th, e_{n+1})}\, (Rg)(\tilde \mu^{-1} \th),  \qquad \th \in \tilde S^{n}_+,\ee
where $d (\th, e_{n+1})$ is the geodesic distance between $\th$ and $e_{n+1}$.
\end{theorem}

Theorems \ref{azw1a2} and \ref{Con22on22} yield the corresponding result for the kernel of the operator $F$.
For $f$  even, it suffices to consider the points
 $\th\in S^n$,  which are represented in the spherical polar coordinates as
\[\th = \eta\, \sin \psi  + e_{n+1} \, \cos \,\psi, \qquad \eta \in \sn, \qquad 0<\psi<\pi/2.\]
The corresponding Fourier-Laplace coefficients (in the $\eta$-variable) have the form
\be\label{TPNY} f_{m,\mu} (\psi)= \intl_{\sn} f(\eta\, \sin \psi  + e_{n+1} \, \cos \,\psi)\, Y_{m,\mu} (\eta)\, d\eta.\ee

\begin{theorem} \label {786NGR1SP}  Suppose that
\be\label{TPNY1del} \intl_{|\th_{n+1}|<\a} |f(\th)| \,  d\th <\infty\quad \mbox{for  all $\;\a\in (0,1)$}\ee
and $(F f)(\th)=0$  a.e. on $S^n$. Then  all the Fourier-Laplace coefficients (\ref{TPNY})
have the form
\be\label{TUUNY1del} f_{m,\mu} (\psi)\stackrel{a.e.}{=}\left\{\begin{array} {ll} 0 &if\quad m=0,1,\\
\sin^{-n} \psi  \sum\limits_{j=1}^{[m/2]} c_{j}\, \cot^{m-2j}\psi  &if\quad m\ge 2,
\end{array} \right.\ee
with  some constants $c_j$. Conversely, for any constants $c_j$ and any  $f$ satisfying (\ref{TPNY1del}) and (\ref{TUUNY1del}), we have $(F f)(\th)=0$  a.e. on $S^n$.
\end{theorem}

\section{The Spherical Slice Transform}\label{Slice}

The spherical slice  transform
\be\label{Sliceint1}
(\frS f) (\gam)=\intl_{\gam} f(\eta)\, d_{\gam}\eta,  \ee
assigns to a function $f$ on  $S^n$ the integrals of $f$ over $(n-1)$-dimensional geodesic spheres   $\gam \subset S^n$  passing through the north pole $e_{n+1}$.
Every geodesic sphere $\gam$  can be indexed  by  its center $\xi=(\xi_1, \ldots, \xi_{n+1})$ in the
the closed  hemisphere
\[
\bar S^n_+=\{\xi=(\xi_1, \ldots \xi_{n+1})\in S^n:\; 0\le \xi_{n+1}\le 1\},\]
so that
 \[\gam \equiv \gam(\xi)=\{\eta \in S^n: \eta \cdot \xi =e_{n+1} \cdot \xi\},    \qquad \xi \in \bar S^n_+.\]
Using spherical polar coordinates, for $\xi \in \bar S^n_+ $ we  write
\[ \xi= \th \sin \psi +e_{n+1}\, \cos \, \psi, \qquad  \th\in \sn \subset \rn, \quad 0\le \psi \le \pi/2,\]
\[\gam \equiv\gam(\xi)\equiv\gam(\th, \psi), \qquad (\frS f) (\gam)\equiv(\frS f) (\xi)\equiv (\frS f) (\th, \psi).\] Then
 \[\gam(\xi)=\{\eta \in S^n: \eta \cdot \xi =\cos \, \psi\}.\]

Consider the bijective mapping
\be\label{MUIT} \rn \ni x \xrightarrow{\;\nu\;} \eta\in S^{n}\setminus \{e_{n+1}\},  \qquad \nu (x)=\frac{2x+(|x|^2-1)\,e_{n+1}}
{|x|^2+1}.\ee
 The inverse mapping $\nu^{-1}: S^{n}\setminus \{e_{n+1}\} \rightarrow \rn$ is the  stereographic projection  from the north pole $e_{n+1}$ onto $\rn=\bbr e_1 \oplus \cdots \oplus \bbr e_{n}$. If
 \[ \eta= \om \,\sin \vp + e_{n+1} \cos \, \vp, \qquad \om \in \sn, \qquad 0<\vp\le \pi,\]
then $x=\nu^{-1} (\eta)=s\om$, $s=\cot (\vp/2)$.

\begin{lemma} \label {stereoviatS} \cite[Lemma 7.2]{Ru14}
The spherical slice transform on $S^n$ and  the hyperplane Radon transform on $\rn$ are linked by the formula
\be \label {stereogr1r3}
(\frS f) (\theta, \psi)=(Rg)(\theta, t),  \qquad t=\cot \psi,\ee
\be \label {stereogr1r4} g(x)=\left (\frac{2}{|x|^2+1}\right)^{n-1}  (f\circ \nu) (x),\ee
provided that either side of (\ref{stereogr1r3}) is finite when $f$ is replaced by $|f|$.
\end{lemma}

To describe the kernel of the operator $\frS$,  we use the spherical harmonic decomposition of
$f(\eta)=f(\om \sin \vp +e_{n+1}\, \cos \, \vp) $ in the $\om$-variable. Let
\be \label {OOgr1r3} f_{m,\mu} (\vp) =\intl_{\sn }f(\om \sin \vp +e_{n+1}\, \cos \, \vp)\, Y_{m,\mu} (\om)\, d\om.\ee
 Then   Theorems \ref{azw1a2} and  Lemma \ref{stereoviatS} imply the following  statement.

\begin{theorem} \label {zasliep}  Suppose that
\be\label{TPNY1delep} \intl_{\eta_{n+1}>1-\e}
\frac{|f(\eta)|}{(1-\eta_{n+1})^{1/2}}\,d\eta<\infty \quad \forall \, \e\in (0, 2], \ee
and $(\frS f)(\xi)=0$ for almost all  $\xi\in \bar S^n_+$. Then  all the Fourier-Laplace coefficients (\ref{OOgr1r3})
have the form
\be\label{zNYUep} f_{m,\mu} (\vp)\stackrel{a.e.}{=}\left\{\begin{array} {ll} \!0 \!&if\; m\!=\!0,1,\\
\!(1\!-\!\cos\, \vp)^{1-n}  \displaystyle{\sum\limits_{j=1}^{[m/2]} c_{j} \left (\tan \frac{\vp}{2}\right )^{n+m-2j}} \! &if\; m\!\ge \!2,
\end{array} \right.\ee
with  some constants $c_j$. Conversely, for any constants $c_j$ and any  $f$ satisfying (\ref{TPNY1delep}) and (\ref{zNYUep}), we have $(\frS f)(\xi)=0$ a.e. on  $\bar S^n_+$.
\end{theorem}

\section{The Totally Geodesic Radon Transform on the Hyperbolic Space}\label {Link wi}

  Let $\bbe^{n, 1}\sim \bbr^{n+1}$, $n\ge 2$, be  the $(n+1)$-dimensional  pseudo-Euclidean   real vector space with the inner product
\be\label {tag 2.1-HYP}[x,y] = - x_1 y_1 - \ldots -x_n y_n + x_{n +1} y_{n +1}. \ee
The real hyperbolic space $\hn$ is realized as the upper sheet of the two-sheeted hyperboloid    in $\bbe^{n, 1}$, that is,
\[\hn = \{x\in \bbe^{n,1} : \,[x,x] = 1, \ x_{n +1} > 0 \}; \]
see \cite{GGV}. The corresponding one-sheeted hyperboloid is defined by
\[\overset {*}{\bbh}{}^n = \{
x \in E^{n,1}: \,[x,x] = - 1 \}.\]
 The totally geodesic Radon transform of a function $f$ on $\hn$ is an integral operator of the form
\be\label {tag 3.1-HYP} (\frR f) (\xi) = \int\limits_{\{x\in\hn:\,[x, \xi] = 0\}}
f (x) \, d_\xi x,  \qquad \xi \in \hns, \ee
 and represents an  even function on $\hns$; see  \cite{BR1, GGV, H11}.

Using the projective equivalence of  the operator (\ref{tag 3.1-HYP}) and  the hyperplane  Radon
transform $R$, as in  \cite[Lemma 8.3]{Ru14}  (see also  \cite {Ku94, BCK}),
 we obtain the following kernel description, which is implied by  Theorem \ref{azw1a2}. We write $x\in \hn$ in the hyperbolic polar coordinates as $x = \theta\, \sh r  + e_{n+1} \, \ch r$, $\th \in \sn$, $r>0$, and consider the Fourier-Laplace coefficients
\be\label{TPNYH} f_{m,\mu} (r)= \intl_{\sn} f(\th\, \sh r + e_{n+1} \ch r)\, Y_{m,\mu} (\th)\, d\th.\ee

\begin{theorem} \label {786NGR1}  Let
\be\label{THHHep} \intl_{x_{n+1}>1+\del}
|f(x)|\, \frac{dx}{x_{n+1}} <\infty \quad \forall \,\del>0, \ee
and let $(\frR f)(\xi)=0$  a.e.  on $\hns$. Then  all the Fourier-Laplace coefficients (\ref{TPNYH})
have the form
\be\label{azw2INHY5} f_{m,\mu} (r)\stackrel{a.e.}{=}\left\{\begin{array} {ll} \!0 \!&if\; m\!=\!0,1,\\
\sh^{-n}  \sum\limits_{j=1}^{[m/2]} c_{j}\, \coth^{m-2j} \psi \! &if\; m\!\ge \!2,
\end{array} \right.\ee
with  some constants $c_j$. Conversely, for any constants $c_j$ and any  $f$ satisfying (\ref{THHHep}) and (\ref{azw2INHY5}), we have $(\frR f)(\xi)=0$  a.e.  on $\hns$.
\end{theorem}

\section{Conclusion} The list of projectively equivalent Radon-like transforms, the kernels of which admit effective characterization using the results of the present paper, can be essentially extended. For example, one can obtain kernel descriptions of the  exterior/interior analogues of the Radon-like transforms (and their duals) from Sections 4-7. We leave this useful exercise to the interested  reader.

\end{document}